\documentclass[10pt]{amsart}

\usepackage{enumerate,url,amssymb,  mathrsfs, graphicx, pdfsync}
\newtheorem{theorem}{Theorem}[section]
\newtheorem{lemma}[theorem]{Lemma}
\newtheorem*{lemma*}{Lemma}

\newtheorem{proposition}[theorem]{Proposition}

\theoremstyle{definition}
\newtheorem{definition}[theorem]{Definition}

\theoremstyle{remark}
\newtheorem{remark}[theorem]{Remark}

\numberwithin{equation}{section}

\newcommand{\abs}[1]{\lvert#1\rvert}

\newcommand{\norm}[1]{\lVert#1\rVert}

\newcommand{\A}{\mathbb{A}}
\newcommand{\B}{\mathbb{B}}

\newcommand{\W}{\mathscr{W}}

\newcommand{\R}{\mathbb{R}}
\newcommand{\X}{\mathbb{X}}

\newcommand{\Y}{\mathbb{Y}}

\newcommand{\onto}{\xrightarrow[]{{}_{\!\!\textnormal{onto\,\,}\!\!}}}
\newcommand{\into}{\xrightarrow[]{{}_{\!\!\textnormal{into\,\,}\!\!}}}
\newcommand{\deff}{\stackrel {\textnormal{def}}{=\!\!=} }
\newcommand{\bydef}{\stackrel{\textnormal{def}}{=\!\!=}}

\DeclareMathOperator{\loc}{loc}

\def\le{\leqslant}
\def\ge{\geqslant}

\begin{document}

\title[Monotone Sobolev Mappings]{\Large{Monotone Sobolev Mappings}\\$\textnormal{of planar domains and surfaces} $}
\author[T. Iwaniec]{Tadeusz Iwaniec}
\address{Department of Mathematics, Syracuse University, Syracuse,
NY 13244, USA and Department of Mathematics and Statistics,
University of Helsinki, Finland}
\email{tiwaniec@syr.edu}

\author[J. Onninen]{Jani Onninen}
\address{Department of Mathematics, Syracuse University, Syracuse,
NY 13244, USA and Department of Mathematics and Statistics, P.O. Box 35 (MaD), FI-40014, University of Jyv\"askyl\"a, Finland}
\email{jkonnine@syr.edu}
\thanks{ T. Iwaniec was supported by the NSF grant DMS-1301558 and the Academy of Finland project 1128331.
J. Onninen was supported by the NSF grant DMS-1301570.}




\begin{abstract} An approximation theorem of Youngs (1948)  asserts that a  continuous map between compact oriented topological 2-manifolds (surfaces) is monotone if and only if it is a uniform limit of homeomorphisms. Analogous approximation of Sobolev mappings is at the very heart of Geometric Function Theory (GFT) and Nonlinear Elasticity (NE). In both theories the mappings in question arise naturally as weak limits of energy-minimizing sequences of homeomorphisms. As a result of this, the energy-minimal mappings turn out to be monotone.
In the present paper we show that, conversely, monotone mappings in the Sobolev space $\,\mathscr W^{1,p}\,, \,1<p < \infty\,$, are none other than $\,\mathscr W^{1,p}\,$-weak (also strong) limits of homeomorphisms. In fact, these are limits of diffeomorphisms. By way of illustration, we establish the existence of energy-minimal deformations within the class of Sobolev monotone mappings for $\,p\,$-harmonic type energy integrals.
\end{abstract}

\maketitle

\section{Introduction} There has been recently increasing interest in the Sobolev $\,\mathscr W^{1,p}\,$- homeomorphisms and their weak and strong limits.  In planar domains (or surfaces) and $\, p\geqslant 2\,$,  these limits turn out to be monotone mappings in the topological sense of C.B. Morrey~\cite{Morrey}, see Definition \ref{DefMonotonicity} below. We shall see that, conversely, every $\,\mathscr W^{1,p}\,$-map that is continuous and monotone between Lipschitz domain (this time for any exponent $\,1 < p < \infty\,$) can be approximated by homeomorphisms uniformly and strongly in the Sobolev norm; hence, by $\mathscr C^\infty$-diffeomorphisms, also by piecewise affine homeomorphisms.  A motivation for Sobolev mappings comes from the
study of extremal problems in Geometric Function Theory (GFT) ~\cite{AIMb, IKKO, IKOhopf, IKOdu, IMb, IOan, IOho}. Further motivation comes from the fields of materials science such as Nonlinear Elasticity (NE) ~\cite{Anb, Bac, Cib, MHb, Sib, TNb}.\\
Let us begin with a brief analysis of the $\,p\,$-harmonic energy of homeomorphisms $\,f : \mathbb X \onto \mathbb Y\,$ between bounded domains $\, \mathbb X , \mathbb Y \subset \mathbb R^2\,$.
\begin{equation}
 \mathscr E_p[f] = \int_\mathbb X |Df(x)|^p\,\textnormal d x\;,\;\;\;\;\;\; \textnormal{where}\;\;|Df(x)|^2\;\bydef \textnormal{Tr} [D^*\!f(x)\, Df(x)]
\end{equation}
Hereafter $\,Df(x)\,$ stands for the Jacobian matrix of $\,f\,$ (deformation gradient).
 Denote by $\mathscr H_p(\mathbb X , \mathbb Y)\,$ the class of orientation preserving homeomorphsms  $\,f : \mathbb X \onto \mathbb Y\,$ of finite  energy and
\begin{equation}
\mathscr E_p (\mathbb X, \mathbb Y)  \bydef  \inf_{ f \in \mathscr H_p(\mathbb X , \mathbb Y) } \int_\mathbb X |Df(x)|^p\,\textnormal d x
\end{equation}
The infimum may or may not be attained. If not, we wish to find a map $\,h_\circ \in \mathscr W^{1,p}(\mathbb X , \mathbb Y)\,$ in close proximity to $\,\mathscr H_p(\mathbb X , \mathbb Y)\,$ such that
\begin{equation}\label{equality}
 \int_\mathbb X |Dh_\circ(x)|^p\,\textnormal d x \,=\, \mathscr E_p (\mathbb X, \mathbb Y)
\end{equation}

It is natural to look for $\,h_\circ\,$  as $\,\mathscr W^{1,p}\,$-weak limit of an energy-minimizing sequence of homeomorphisms $\,h_j : \mathbb X \onto \mathbb Y\,$. That this indeed would solve the problem (at least for Lipschitz domains and $\,p\geqslant 2\,$) is far from being obvious. We always have the upper bound
\begin{equation}\label{energybound}
 \int_\mathbb X |Dh_\circ(x)|^p\,\textnormal d x  \leqslant \mathscr E_p (\mathbb X, \mathbb Y)
\end{equation}
Equality occurs if and only if the minimizing sequence $\,h_j \rightharpoonup h_\circ\,$ actually converges strongly in $\,\mathscr W^{1,p}(\mathbb X, \mathbb R^2)\,$, which places $\,h_\circ\,$ in the closest possible proximity to $\,\mathscr H_p(\mathbb X , \mathbb Y)\,$.
Equality (\ref{equality}) is of practical significance.  Indeed, once $\,h_\circ\,$ fails to be injective, it will tell us when to stop the minimizing sequence of homeomorphisms; that is, prior to the conditions favorable for the collapse of injectivity. This is a phenomenon known as \textit{interpenetration of matter}.\\

  \textbf{Monotone Mappings.} The concept is due to C.B. Morrey \cite{Morrey}  (1935). The interested reader is referred to the Proceedings of the Conference on Monotone and Open Mappings \cite{Mc} (1970) and the series of early papers by G.T. Whyburn for further reading.
\begin{definition}\label{DefMonotonicity}
Let $\,\mathbf A\,$ and $\,\mathbf B\,$ be compact metric spaces.
A continuous map  $\,h \colon \mathbf A \onto \mathbf B\,$ is said to be monotone if every fiber $\,h^{-1}\{b\} \subset \mathbf A\,$ of a point $\,b \in \mathbf B\,$ is connected. As shown by G.T. Whyburn  the preimage $\,h^{-1} (\mathcal C) = \{a\in \mathbf A \colon h(a) \in \mathcal C\}\,$ of any connected set $\,\mathcal C \subset \mathbf B\,$ is connected in $\,\mathbf A\,$.
\end{definition}
Let us emphasize that monotone mappings are continuous, by the definition.\\

 While manifolds and Riemannian metric tensors are not the primary issues, it is desirable to keep them in mind since the topological aspects  really crystalize in the manifold setting.
    Thus we choose and fix, as \textit{reference manifolds}, two $\,\mathscr C^1\,$-smooth closed (compact without boundary) oriented Riemannian 2-manifolds $\,\mathscr X\,$ and $\,\mathscr Y\,$ of the same topological type.  We shall consider multiply connected \textit{Jordan domains} $\,\mathbb X\subset \mathscr X\,$ and $\,\mathbb Y \subset \mathscr Y\,$. Precisely, $\,\mathbb X\,$ and $\,\mathbb Y\,$  will be  obtained by removing from $\,\mathscr X\,$ and $\,\mathscr Y\,$ the same number, say $\,0 \leqslant \ell  < \infty\,$, of closed disjoint topological disks. In fact, any pair of topologically equivalent open $\,\mathscr C^1\,$-smooth surfaces with $\,\ell\,$ boundaries can be obtained in this way. Such are planar multiply connected Jordan domains in $\,\widehat{\mathbb R}^2 \simeq \mathbb S^2 \subset \mathbb R^3\,$.

\begin{theorem}[Youngs, \cite{Yo1}] Let $\,\mathbb X \subset \mathscr X\,$ and $\, \mathbb Y \subset \mathscr Y\,$  be Jordan domains of the same topological type. A map $\,h : \overline{\mathbb X} \onto \overline{\mathbb Y}\,$ is monotone if and only if it is a uniform limit of  homeomorphisms $\,h_j : \overline{\mathbb X} \onto \overline{\mathbb Y}\,, j = 1, 2, ... \,$.\\ It follows, in particular, that the boundary mapping $\,h : \partial\mathbb X \onto \partial\mathbb Y\,$ is monotone as well.
\end{theorem}$\,$\\

\textbf{Sobolev Variant of Youngs' Approximation Theorem.}
The main result:
\begin{theorem}\label{maintheorem} Let $\,\mathbb X \subset \mathscr X\,$ and $\, \mathbb Y \subset \mathscr Y\,$  be Jordan domains of the same topological type, $\,\mathbb Y\,$ being Lipschitz.
For every monotone map  $\,h : \overline{\mathbb X} \onto \overline{\mathbb Y}\,$ in the Sobolev space $\,\mathscr W^{1,p}(\mathbb X, \mathbb Y)\, , 1< p< \infty\,$, there  exists a sequence of monotone mappings $\,h_j : \overline{\mathbb X} \onto \overline{\mathbb Y}\,$ such that:

\begin{itemize}
\item [(i)] $\,h_j :  \mathbb X \onto \mathbb Y\,$\; are homeomorphisms
\item [(ii)] $\,h_j \rightrightarrows h \,$ uniformly  on $\,\overline{\mathbb X}\,$
\item [(iii)] $\,h_j \rightarrow h \,\,$ strongly in  $\,\mathscr W^{1,p}(\mathbb X, \mathbb Y)\,$
\item [(iv)] $\,h_j = h \,: \partial \mathbb X \onto \partial \mathbb Y\, , \;for \;j = 1, 2, ...\;$
\end{itemize}
\end{theorem}

Let  $\mathscr M_p\,(\overline{\mathbb X} , \overline{\mathbb Y})\,$ denote the class of orientation preserving monotone mappings $\,f : \overline{\mathbb X} \onto \overline{\mathbb Y}\,$ of finite  $\,p\,$-harmonic energy\,,$\,1 < p < \infty\,$.  Theorem \ref{maintheorem} implies (by \textit{direct method})  that there always exists $\, h \in \mathscr M_p\,(\overline{\mathbb X} , \overline{\mathbb Y})\,$ with smallest $\,p\,$-harmonic energy. More importantly, the energy of $\,h\,$ equals precisely the infimum of the energy among homeomorphisms:

\begin{equation}\nonumber
 \int_\mathbb X |Dh(x)|^p\,\textnormal d x  = \min_{ f \in \mathscr M_p(\overline{\mathbb X} , \overline{\mathbb Y}) } \int_\mathbb X |Df(x)|^p\,\textnormal d x = \inf_{ f \in \mathscr H_p(\mathbb X , \mathbb Y) } \int_\mathbb X |Df(x)|^p\,\textnormal d x = \mathscr E_p\,( \mathbb X , \mathbb Y)
\end{equation}
see Section \ref {Application} for a definition of the $\,p\,$-harmonic integrals on surfaces.
In other words, no \textit{Lavrentiev Phenomenon} occurs in the class of monotone Sobolev mappings.

\begin{remark}
It is worth noting that, for an arbitrary pair $\,(\mathbb X , \mathbb Y)\,$  of topologically equivalent planar domains, diffeomorphisms are dense in  $\,\mathscr H_p(\mathbb X , \mathbb Y)\,$, see \cite{IKO2}. Thus one can take for $\,h_j\,$ in Theorem \ref{maintheorem} a sequence of $\,\mathscr C^\infty\,$-smooth diffeomorphisms. Furthermore in case $\,p = 2\,$, if
one is willing to sacrifice the boundary condition (iv)\, then the diffeomorphisms  $\,h_j :  \mathbb X \onto \mathbb Y\,$ can be chosen to be homeomorphisms up to the boundary, again denoted by $\,h_j :  \overline{\mathbb X} \onto \overline{\mathbb Y}\,$, see \cite{IKO3}.
\end{remark}
\begin{remark} It is not difficult to see that, in case of planar Lipshitz domains, a sequence of homeomorphisms $\,h_j :  \mathbb X \onto \mathbb Y\,$ which converges weakly in  $\,\mathscr W^{1,p}(\mathbb X, \mathbb Y)\, $, $\,p \geqslant 2\,$, actually converges uniformly to a monotone mapping $\,h :  \overline{\mathbb X} \onto \overline{\mathbb Y}\,$. Thus, in particular,  Theorem \ref{maintheorem} implies that $\,h\,$ can be realized as $\,\mathscr W^{1,p}$-strong limit of homeomorphisms. We therefore recover a result in ~\cite{IOw=s}; accordingly, weak sequential closure and strong closure of  $\,\mathscr H_p(\mathbb X, \mathbb Y)\, $, $\,p \geqslant 2\,$,  are the same.
\end{remark}

\section{Topological Preliminaries}
This section is intended as a gentle introduction to underlying 
 geometric analysis on planar domains and surfaces.
   When discussing the Sobolev class $\,\mathscr W^{1,p}(\mathbb X, \mathbb Y)\,$, it is particularly  convenient to embed the reference manifolds $\,\mathscr X \,$ and $\, \mathscr Y\,$ into an Euclidean space. We may, for example, use  the Nash-Kuiper embedding theorem~\cite{Ku, Na} which assures that the oriented 2-manifolds are isometrically $\,\mathscr C^1$ -embeddable in $\,\mathbb R^3\,$.  Thus we may (and do) assume that
$$\,\mathscr X , \mathscr Y \subset \mathbb R^3\,\;;\;\;\;\textnormal{inclusion being a  $\,\mathscr C^1\,$- embedding} .$$
Note that for the purpose of defining $\,\mathscr W^{1,p}(\mathbb X, \mathbb Y)\,$ one needs only assume uniform upper and lower bounds on the metric tensors.\\
When interpreting the conclusions, one might view the surfaces $\,\mathbb X\,$ and $\, \mathbb Y\,$ as thin films in $\,\mathbb R^3\,$, or flat plates in case $\,\mathscr X =\mathscr Y = \widehat{\mathbb R^2} \simeq \mathbb S^2 \subset \mathbb R^3\,$.\\
The components of $\,\mathscr X\setminus\mathbb X\,$ and $\,\mathscr Y \setminus \mathbb Y\,$ are topological disks, say
$$\,\mathscr X\setminus\mathbb X\, = \mathbb X_1 \cup \mathbb X_2 \cup ... \cup \mathbb X_\ell \;\;\textnormal{and}\;\;\; \mathscr Y\setminus\mathbb Y\, = \mathbb Y_1 \cup \mathbb Y_2 \cup ... \cup \mathbb Y_\ell $$
 Their boundaries, denoted by   $\,\partial \mathbb X_\nu \bydef \mathfrak X_\nu\,$ and $\,\partial \mathbb Y_\nu \bydef \Upsilon_\nu\,$,  are exactly the components of $\,\partial \mathbb X\,$ and $\,\partial \mathbb Y\,$,
$$
 \;\partial \mathbb X =   \mathfrak X_1 \cup  ... \cup  \mathfrak X_\ell \;\;\textnormal{and} \;\; \; \partial \mathbb Y =   \Upsilon_1 \cup  ... \cup  \Upsilon_\ell\; .
$$
 By convention,  $\, \mathbb X\,$ and  $\,\mathbb Y\,$ have no boundary if $\,\ell = 0\,$, in which case $\,\mathbb X = \mathscr X\,$ and $\,\mathbb Y = \mathscr Y\,$.

\subsection*{Continuous Functions and Homeomorphisms} We shall work with various function spaces defined on subsets of the reference manifold $\,\mathscr X\,$:
\begin{itemize}
\item  For a compact subset $\,\mathbb A \subset \mathscr X\,$  we denote by $\,\mathscr C (\mathbb A)\,$  the space of continuous functions $\,h \colon \mathbb A \to \R^3\,$ furnished with the norm:
\[\norm{\,h\,}_{\mathscr C (\mathbb A)}  = \max_{x \in \mathbb A} \,\abs{h(x)}\]
The notation  $\,h \in \mathscr C (\mathbb A, \mathbb B)\,$ will be used if we want to make explicit
the range of $\,h : \mathbb A \rightarrow \mathbb B \subset \mathbb R^3\,$.\\
\item $\mathscr H (\X , \Y)\,$  consists of orientation preserving homeomorphism $h \colon \X \onto \Y$.\\
Every homeomorphism $\,h : \mathbb X \onto \mathbb Y\,$ gives rise to a one-to-one correspondence between boundary components  $\,\mathfrak X_1 ,\, \mathfrak X_2 , ...\,, \mathfrak X_\ell\,$  and $\,\Upsilon_1 ,\, \Upsilon_2 , ...\,, \Upsilon_\ell\,$ by means of cluster limits. We conveniently rearrange the indices so that the boundary correspondence reads as follows:
\begin{equation}
h : \mathfrak X_\nu \rightsquigarrow \Upsilon_\nu\;\,,\;\;\;\; \textnormal{for}\;\; \nu = 1,... , \ell\;;
\end{equation}
This arrangement will tacitly be assumed throughout this paper for homeomorphisms in the class $\,\mathscr H (\X , \Y)\,$ and their uniform limits.\\
\item $\,\mathscr H (\widehat{\X},\widehat{\Y}) \,$ consists of homeomorphisms $\,h \colon \X \onto \Y\,$ which extend continuously to the closure of $\,\mathbb X\,$. Note that the continuous extensions become monotone maps from $\,\overline{\mathbb X}\,$ onto $\,\overline{\mathbb Y}\,$, still denoted by $\,h \colon \overline{\X} \onto \overline{\Y}\,$. They take $\partial \overline{\X}$ onto $\partial \overline{\Y}\,$; specifically, $\,h(\mathfrak X_\nu) = \Upsilon_\nu\,$ for $\,  \nu = 1, ... , \ell\,$. We refer to $\,h \colon \mathfrak X_\nu \onto \Upsilon_\nu\,, \; \nu = 1, ... , \ell\, $,  as the boundary maps. Each of these boundary maps is monotone.\\
\item $\mathscr H (\overline{\X} ,  \overline{\Y})  \,$ is the space of homeomorphisms $\,h \colon \overline{\X} \onto \overline{\Y}$.
\end{itemize}\vskip0.1cm
 Similar notation, with the obvious analogous meaning, will be used for spaces of mappings defined on other  subsets of $\, \mathscr X\,$.

\subsection*{Monotone Mappings Versus Uniform Limits of Homeomorphisms} Let $\,\A \subset \mathscr X\,$ and $\,\B \subset \mathscr Y\,$ be compact.  The class\\
 $$\mathscr M (\A, \B) \subset \mathscr C (\A, \B)\,\; \textnormal{stands for the space of monotone mappings}\, \,h \colon \A \onto \B\,.$$\\

Obviously $\,\mathscr H(\mathbb A , \mathbb B) \varsubsetneq \mathscr M(\mathbb A , \mathbb B)\,$. We now invoke the Approximation Theorem of J. W. T. Youngs~\cite{Yo1}.
\begin{theorem}\label{HomeomorphicApproximation}
Let $\,\overline{\X}\,$ and $\,\overline{\Y} \,$ be topologically equivalent compact 2-manifolds.

Then for every monotone map $\,h \colon \overline{\X} \onto \overline{\Y}\,$ there exists a sequence of homeomorphisms $\,h_j \colon \overline{\X} \onto \overline{\Y}\,$ converging uniformly to $\,h$. In symbols,
$$\,{\mathscr M} (\overline{\X}, \overline{\Y}) = \overline{\mathscr H (\overline{\X}, \overline{\Y})} \,$$
\end{theorem}
This theorem will provide us with a powerful tool when dealing with monotone mappings. We shall appel to it repeatedly to either recover or refine the well known properties of monotone mappings betwen  Jordan domains $\,\mathbb X \subset \mathscr X\,$ and $\,\mathbb Y\subset \mathscr Y \,$.

 For example, Theorem \ref{HomeomorphicApproximation}  readily   implies that
  \begin{lemma}

We have the inclusions
$$\mathscr H (\overline{\X} ,  \overline{\Y}) \; \varsubsetneq\; \mathscr H (\widehat{\X} ,\widehat{ \Y})\; \varsubsetneq \;\mathscr M (\overline{\X} ,  \overline{\Y})\; = \; \overline{\mathscr H (\overline{\X}, \overline{\Y})} \;\varsubsetneq\; \mathscr C(\overline{\X} ,  \overline{\Y})\,$$
Moreover, for $\, h \in \mathscr M (\overline{\X}, \overline{\Y})\,$,  each boundary map $\,h : \partial \mathfrak X_\nu  \onto \partial \Upsilon_\nu\,$ is monotone.
\end{lemma}
We shall also take advantage of the following refinement of the \textit{Modification Theorem} in \cite{Yo1}.
\begin{lemma}[Homeomorphic Extension of a Monotone Boundary Map]\label{HomeomorphicExtensionInsideCells}
Let $\,\mathbb X_\circ \subset \mathscr X\,$ and $\,\mathbb Y_\circ \subset \mathscr Y\,$ be  simply connected Jordan domains. Then every monotone map $\,h : \partial \mathbb X_\circ \onto \partial \mathbb Y_\circ\,$ admits a continuous monotone extension $\,h : \overline{\mathbb X_\circ\!\!} \onto  \overline{\mathbb Y_\circ\!\!}\,\,$. Such an extension can be further modified to become a homeomorphism between $\,\mathbb X_\circ\,$ and $\,\mathbb Y_\circ\,$.
\end{lemma}
\begin{proof}
 First, with the aid of the uniformization theorem, we transform $\,\overline{\mathbb X_\circ\!\!} \,$ and $\,\overline{\mathbb Y_\circ\!\!} \,$ homeomorphically onto the closed Euclidean disks.  We obtain a monotone map between circles, which  we extend (in a radial fashion) to a monotone map of the disks.  We then lift such an extension back to the reference  manifolds $\,\mathscr X\,$ and $\,\mathscr Y\,$, completing the proof of the first statement. Then, by "\textit{Modification Theorem}" in \cite{Yo1}, this extension can be modified to become a homeomorphism between $\,\mathbb X_\circ\,$ and $\,\mathbb Y_\circ\,$.
\end{proof}
 \begin{remark} In Theorem \ref{HomeomorphicApproximation}, if one is willing to forgo the univalence of the boundary mappings $\,h_j \colon \partial \X \,\onto \partial \Y\,$, but only wants them to be injective from  $\,\mathbb X \onto \mathbb Y\,$, then we can ensure that    $\, h_j = h\,$ on $\,\partial \X\,$. However, this variant of Youngs' theorem will not be exercised here.
 \end{remark}

\subsection*{Monotone Extension Outside $\,\mathbb X\,$.}
\begin{lemma}
Every $\,h \in {\mathscr M} (\overline{\X}, \overline{\Y}) \,$ extends as a continuous monotone map between reference manifolds.
\end{lemma}

\begin{proof}
If $\overline{\X}$ has no boundary then so does $\overline{\Y}\,$. Hence $\overline{\X} = \mathscr X \,$ and $\,\overline{\Y} = \mathscr Y\,$. Now, consider a boundary component $\,\mathfrak X_\nu \subset \partial \X\,$ and the corresponding boundary component $\,\Upsilon_\nu \subset \partial \Y\,$. Recall that  $\,h \colon \mathfrak X_\nu \onto \Upsilon_\nu\,$ is monotone.  Both $\,\mathfrak X_\nu\,$ and $\,\Upsilon_\nu\,$ are topological circles on the reference surfaces $\,\mathscr X\,$ and $\,\mathscr Y $, respectively.
Filling these circles with topological open disks, say  $\,{\mathbb X_\nu}\subset \mathscr X \,$ and $\, {\mathbb Y_\nu}\subset \mathscr Y\,$, results in closed 2-cells in the reference manifolds. The extension is immediate from Lemma \ref{HomeomorphicExtensionInsideCells}.
\end{proof}
\begin{remark}
With little extra efforts one can ensure that such an extension is a homeomorphism between the components $\,\mathbb X_\nu\,$ and $\,\mathbb Y_\nu\,$, for all  $\nu = 1,2,..., \ell\,$. But there will be no need for this.
\end{remark}
We continue to use the same notation   $\, h : \mathscr X \onto \mathscr Y\,$ for the monotone extension. Thus  $\,h\,$ becomes an element of $\,\mathscr M(\mathscr X , \mathscr Y)\,$ as well. There will be no need of any regularity  of $\, h : \mathscr X \onto \mathscr Y\,$ outside  $\,\mathbb X\,$ .

\subsection*{Cells and Pre-cells} A cell in the reference manifold $\,\mathscr Y\,$ is any simply connected Jordan domain $\,\mathcal Q \subset \mathscr Y\,$.  Its preimage $\,h^{-1}(\mathcal Q)\subset \mathscr X \,$ under a monotone map $\,h : \mathscr X \onto  \mathscr Y\,$ is still simply connected but not necessarily a Jordan domain. We refer to $\,h^{-1}(\mathcal Q)\subset \mathscr X \,$   as \textit{pre-cell} in $\,\mathscr X\,$. Beware that not every simply connected domain in $\,\mathscr X\,$ comes  as a pre-cell for $\,h\,$;  and that, in general, the closure of a pre-cell may not contain $\,h^{-1}(\overline{\mathcal Q})\,$. In fact, the strict inclusions $\,\overline{h^{-1}(\mathcal Q)} \varsubsetneq h^{-1}(\overline{\mathcal Q})\,$ and $\,\partial h^{-1}(\mathcal Q) \varsubsetneq h^{-1}(\partial\mathcal Q)\,$ are typical of monotone mappings. Associated with $\,\mathbb X \subset \mathscr X\,,\; \mathbb Y\subset \mathscr Y\,$ and the monotone map
$\,h: \overline{\mathbb X} \onto  \overline{\mathbb Y}\,$ are the concepts of internal and boundary cells.

\begin{definition}
 The term cell in a domain $\,\mathbb Y\subset \mathscr Y \,$ refers to any simply connected Jordan domain $\,\mathcal Q \subset \mathbb Y\,$ which satisfies one of the following conditions:
\begin{itemize}
\item  $\,\mathcal Q \Subset \mathbb Y\,$, we call it  an \textit{internal cell}\,.
\item  $\,\overline{\mathcal Q} \cap \partial \mathbb Y = \overline{ \mathcal C} \,$ is a closed Jordan arc (not a single point).   We refer to such  $\,\mathcal Q\,$  as \textit{boundary cell} in $\,\mathbb Y\,$ and to $\overline{\mathcal C }$ as its  \textit{external face}. In this case we shall also make use of  the set $\,\mathcal Q_+ \bydef \mathcal Q \cup \mathcal C\,$.
    \end{itemize}
\end{definition}
Every boundary cell $\,\mathcal Q \subset \mathbb Y\,$ can be extended beyond its external face to become a cell in the reference manifold $\,\mathscr Y\,$. Let $\,\mathcal Q^*\subset \mathscr Y\,$  be such an extension so  $\,\mathcal Q = \mathcal Q^* \cap \mathbb Y\,$. Here $\,\partial \mathbb Y\,$  splits $\,\mathcal Q^*\,$  into two simply connected Jordan domains. \\

To every cell $\,\mathcal Q\,$ in $\,\mathbb Y\,$ (internal or boundary)  there corresponds a simply connected domain in $\,\mathbb X\,$, called \textit{pre-cell} in $\,\mathbb X\,$.  It is defined by the following rule.
\begin{itemize}
 \item   If $\,\mathcal Q \Subset \mathbb Y\,$, then we call $\,\mathcal U \deff h^{-1}(\mathcal Q) \Subset \mathbb X\,$ \textit{the internal pre-cell} in $\,\mathbb X\,$.
 \item   If $\,\overline{\mathcal Q} \cap \partial \mathbb Y \neq\emptyset\,$, then we call $\,\mathcal U \deff h^{-1}(\mathcal Q^*) \cap \mathbb X \,$ the \textit{boundary pre-cell} in $\,\mathbb X\,$. It is  of no importance which extension  $\,\mathcal Q^*\,$ is taken in this formula for $\mathcal U$. In fact, we always have  $\,\mathcal U =  h^{-1}(\mathcal Q_+) \cap \mathbb X \,$.
\end{itemize}
The reader is cautioned that in general $\,\mathcal U\,$ is not the preimage of the boundary cell $\,\mathcal Q \subset \mathbb Y\,$; we have only the inclusion  $\,h^{-1}(\mathcal Q) \subset  \mathcal U\,$.
Furthermore, the pre-cells in $\,\mathbb X\,$ need not be Jordan domains, and that is why  some complications (topological and analytical) are to be expected.

\begin{lemma}\label{Precells} The pre-cells (both internal and boundary) are simply connected domains in $\,\mathbb X\,$.
\end{lemma}
\begin{proof} We again take advantage of Youngs' Approximation Theorem.  The lemma holds if $\,h : \mathscr X \onto \mathscr Y\,$ is a homeomorphism. But it also holds if $\,h\,$ is monotone; just use the sequence of homeomorphisms converging uniformly to $\,h\,$, to  represent the pre-cell as union of an increasing sequence of simply connected domains.
\end{proof}

Our next topological fact, which actually strengthens Lemma \ref{HomeomorphicExtensionInsideCells}, will  require some work. Recall that we have extended  $\,h \in {\mathscr M} (\overline{\X}, \overline{\Y}) \,$ to  a map in $\, {\mathscr M} (\mathscr X, \mathscr Y) \,$.
\begin{lemma}[Homeomorphic Replacement in Pre-cells] \label{Homeomorphic replacement}
Let $\,\mathcal Q \subset \mathbb Y\,$ be a cell in $\,\mathbb Y\,$ (internal or boundary) and $\,\mathcal U \subset \mathbb X\,$ the corresponding  pre-cell in $\,\mathbb X\,$.
 Then there exists a monotone map $\,h_{_{\mathcal U}}  \in {\mathscr M} (\mathscr X , \mathscr Y) \,$ such that
\begin{itemize}
\item $\,h_{_{\mathcal U}} \, : \mathcal U \onto \,\mathcal Q\,$ is a homeomorphism
\item $\,h_{_{\mathcal U}} \equiv h \,:\,\mathscr X \setminus \mathcal U\, \onto \mathscr Y \setminus \mathcal Q\,$
\end{itemize}
\end{lemma}
\begin{proof} The proof  comes down to monotone mappings between 2-spheres and a result by T. Rad\'{o} \cite {RadoB} page 66, II.1.47. Let us consider two cases.\\

{\bf Case 1} [\,internal pre-cell\,]\,.
  Suppose $\,\mathcal U = h^{-1} ( \mathcal Q)\,$ where  $\,\mathcal Q \Subset \mathbb Y\,$.  Choose and fix slightly larger cell $\,\mathcal Q \Subset \mathcal Q{\,} ' \Subset \mathbb Y\,$, which gives us a larger pre-cell  $\,\mathcal U \Subset \mathcal U{\,} ' \deff h^{-1} ( \mathcal Q{\,} ')\Subset \mathbb X\,$. We view $\,\mathcal U{\,}'\,$ and $\,\mathcal Q{\,}'\,$ as open simply connected Riemann surfaces. Note that the
map $\,h : \overline{\mathcal U{\,}'} \rightarrow \overline{\mathcal Q{\,}'}\,$ is  continuous and  takes $\,\partial\, \mathcal U{\,}'\,$ into $\,\partial \mathcal Q{\,}'\,$.

Denote by  $\,\widehat{\mathcal U{\,}'}\,$  and $\,\widehat{\mathcal Q{\,}'}\,$ the Alexandroff one-point compactifications of $\mathcal U'$ and $\mathcal Q'$, respectively . These are topological 2-spheres. The unique  continuous extension $\, \widehat{h} : \,\widehat{\mathcal U{\,}'}\,\onto\,  \widehat{\mathcal Q{\,} '}\,$  of $\, h : \,\mathcal U{\,}'\,\onto\,  \mathcal Q{\,} '\,$ remains monotone. At this point we may appeal to \cite {RadoB} page 66, II.1.47,  which asserts that there is a monotone map $\,h_{_\mathcal U} :  \widehat{\mathcal U{\,} '} \onto \widehat{\mathcal Q{\,} '}\,$  which takes $\,\mathcal U\,$ homeomorphically onto $\,\mathcal Q\,$ and agrees with $\,h\,$ outside $\, \mathcal U\,$.\\

{\bf Case 2} [\,boundary pre-cell\,]\,. Suppose $\,\mathcal U \deff h^{-1}(\mathcal Q^*) \cap \mathbb X\,$, where $\,\mathcal Q^*$ is a cell in $\mathscr Y$ such that $\,\mathcal Q = \mathcal Q^* \cap \mathbb Y\,$ and  $\,\mathcal Q^* \cap \partial \mathbb Y \deff \mathcal C \,$  is an open Jordan arc. Denote by $\,\Omega = h^{-1}(\mathcal Q^*)\,$ the pre-cell in $\,\mathscr X\,$ under the map  $\,h : \mathscr X \onto \mathscr Y\,$. As in the previous case we find a monotone mapping $\,H \deff h_\Omega : \mathscr X \onto \mathscr Y\,$ that agrees with $\,h : \mathscr X \setminus \Omega  \;\onto \mathscr Y \setminus \mathcal Q^*\,$  and takes $\,\Omega^*\,$ homeomorphically onto $\,\mathcal Q^*\,$. The issue is that  $\,H : \Omega \onto \mathcal Q^*\,$ need not take $\,\mathcal U = \Omega \cap \mathbb X \,$ onto $\,\mathcal Q^* \cap \mathbb Y\,$ and, even if it does, need not coincide with $\,h\,$ outside $\,\mathcal U\,$.
The idea is to correct $\,H\,$ within $\,\Omega\,$. For, we look closely at the crosscut of $\,\Omega\,$ by  the boundary of $\,\mathbb X\,$, say by the component $\,\mathfrak X = \mathfrak X_\nu \subset \partial \mathbb X\,$, for some $\,\nu = 1, ... , \ell\,$. Let $\,\Upsilon = \Upsilon_\nu \,$ denote the corresponding boundary component of $\,\partial \mathbb Y\,$. Recall that the boundary map $\,h : \mathfrak X \onto \Upsilon\,$ is also monotone; that is, the preimages of connected sets in $\, \Upsilon\,$  are connected in $\,\mathfrak X\,$. In this way we have defined an open subarc of $\,\mathfrak X\,$,
$$
\Gamma = \{ x \in \mathfrak X \, ; \,  h(x) \in \mathcal C\,\}.\,\;\;\;\textnormal{Note that} \;\Gamma \varsubsetneq \mathfrak X\,,\;\;\textnormal{since}\;\; \mathcal C \varsubsetneq \Upsilon = h(\mathfrak X)\,.
 $$
 This subarc has two endpoints (the limit points), say  $\,a, b \in \partial \Omega\,$. It is a topological folklore that an open Jordan arc in a simply connected domain $\Omega$, whose endpoints lie $\partial \Omega$, splits $\Omega$ into two simply connected subdomains. These subdomains are:
 $$
  \mathcal U = \Omega \cap \mathbb X \;\; \textnormal{and} \;\; \mathcal V \bydef \Omega \setminus \overline{\mathbb X}\;,\;\;\;\;\textnormal{hence a decomposition}\;\;\; \Omega = \mathcal U \cup \Gamma \cup \mathcal V
 $$
 Homeomorphism $\, H : \Omega \onto \mathcal Q^*\,$  yields a decomposition of the cell $\,\mathcal Q^*\subset \mathscr Y\,$; namely,
 $$
 \mathcal Q^* = \mathcal U_{_H} \cup \Gamma_{_H} \cup \mathcal V_{_H}\;\;,\;\;\textnormal {where}\;\; \mathcal U_{_H} = H(\mathcal U) \;, \;\, \Gamma_{_H} = H(\Gamma) \; \;\; \textnormal{and}\;\;\; \mathcal V_{_H} = H(\mathcal V)
 $$
 The open Jordan arcs $\,\mathcal C \subset \mathcal Q^*\,$ and $\, \Gamma_{_H} \subset \mathcal Q^*\,$ share common endpoints which we denote  by $\,A = h(a) = H(a) \in \partial \mathcal Q^* \,$ and $\,B = h(b) = H(b) \in \partial \mathcal Q^*\,$.
 We note that $\,\mathcal U_{_H}\,$ and $\, \mathcal V_{_H}\,$ are Jordan domains for which  $\,\Gamma_{_H} \cup \{A,B\}\,$ constitutes their common boundary. Precisely, we have
 $$
 \partial \,\mathcal U_{_H} = \alpha \cup \Gamma_{_H} \;\;,\;\; \partial\, \mathcal V_{_H} = \beta \cup \Gamma_{_H} \;\;,\;\; \partial \,\mathcal U_{_H} \cap \partial\, \mathcal V_{_H} = \overline{\Gamma_{_H}\!\!\!}
 $$
 where $\,\alpha\,$ is the closed sub-arc of $\,\partial \mathcal Q^*\,$ between $\,A\,$ and $\,B\,$ that lies in $\,\overline{\mathbb Y}\,$,  and  $\,\beta \,$ is the closed sub-arc of $\,\partial \mathcal Q^*\,$ between $\,A\,$ and $\,B\,$ that lies in $\mathscr Y \setminus \mathbb Y\,$.\\
 Now,  just the fact that $\, h : \overline{\Omega} \into \overline{\mathcal Q^*}\,$ is continuous and agrees with $\,H\,$ on $\,\partial \Omega \,$ lets us observe that the mapping $\, h \circ H^{-1}  : \mathcal Q^* \onto \mathcal Q^*\,$ (not necessarily injective) extends continuously as a monotone map of $\,\overline{\mathcal Q^*}\,$ onto itself. Upon such an  extension, the boundary map $\, h \circ H^{-1}  : \partial \mathcal Q^* \onto \partial \mathcal Q^*\,$ becomes the identity. Let us see how this extension acts on $\,\partial\, \mathcal U_{_H}\,$. It is still the identity map on $\,\alpha \subset \,\partial \mathcal Q^* \,$ and it takes  the closed subarc $\,\overline{\Gamma_{_{\!H}}\!\!} \subset \partial\, \mathcal U_{_H}\,$ monotonically onto  $\, \overline{\mathcal C} \subset \partial \mathcal Q^*\,$. Thus we have a monotone map
$\, h \circ H^{-1} :  \partial \,\mathcal U_{_H} \onto \partial (\mathcal Q^* \cap \mathbb Y) \,$. It is important to observe  that both $\,\mathcal U_{_H}\,$ and  $\,\mathcal Q^* \cap \mathbb Y\,$ are Jordan domains.  At this stage we appeal to  Lemma  \ref{HomeomorphicExtensionInsideCells} . Accordingly, we  extend $\, h \circ H^{-1} :  \partial \mathcal U_{_H} \onto \partial (\mathcal Q^* \cap \mathbb Y) \,$ continuously, and as a homeomorphism inside the curves. Denote the extension by  $\, (h \circ H^{-1})^\sharp :  \mathcal U_{_H} \onto \mathcal Q^* \cap \mathbb Y \,$. In summary, we have constructed a continuous monotone map $\, F : \mathscr Y \onto\mathscr Y\,$,
\begin{displaymath}
F= \left \{\begin{array}{ll}
\textnormal{identity}  & \textrm{in}\,\, \mathscr Y \setminus \mathcal Q^*  \\
h \circ H^{-1} & \textrm{in}\; \mathscr V_{_H} \\
(h \circ H^{-1})^\sharp & \textrm{in}\; \mathscr U_{_H}
\end{array} \right.
\end{displaymath}
The composition $\, h_{_\mathcal U} \deff\,F \circ\, H :  \mathscr X \onto \mathscr Y\,$ is the desired replacement, as claimed in Lemma \ref{Homeomorphic replacement}.
\end{proof}

\section{Analytical Requisites}
Since the reference manifolds $\,\mathscr X\,$ and $\,\mathscr Y\,$ are closed oriented Riemannian 2-manifolds of class $\,\mathscr C^1\,$, we may  speak of the Sobolev class $\,\mathscr W^{1,p}(\Omega, \mathscr Y)\,$ of mappings $\,h :\Omega \rightarrow \mathscr Y\,$ defined on any open subset  $\,\Omega \subset \mathscr X\,$. We do not reserve any particular notation of the metric tensors on $\,\mathscr X\,$ and $\,\mathscr Y\,$, though we fix them for the rest of this paper. The volume element on $\,\mathscr X\,$, denoted by $\,\textnormal d x\,$,  is the one induced by the metric tensor. We recall the
$C^1$-isometric embeddings  $\,\mathscr X \subset \mathbb R^3\,$ and $\,\mathscr Y \subset \mathbb R^3\,$.
\subsection{Sobolev Mappings Between Surfaces}
The Sobolev space $\,\mathscr W^{1,p}(\Omega)\,$ of real-valued functions on $\,\Omega \subset \mathscr X\,$ will be  endowed  with the seminorm
$$
|\!|\phi |\!|_{\mathscr W^{1,p}(\Omega)}\;\deff\; \left(\int_\Omega |D\phi(x)|^p \,\textnormal d x \right )^{\frac{1}{p}}\;,\;\; 1\leqslant p < \infty
$$
where $\,|D\phi(x)|\,$, defined almost everywhere,  stands for the norm of the linear tangent map $\,D\phi(x) : \mathbf T_x(\Omega) \rightarrow \mathbb R\,$ with respect to the inner product in $\,\mathbf T_x(\Omega)\,$. The Sobolev class $\,\mathscr W^{1,p} (\Omega, \mathscr Y) \subset \mathscr W^{1,p} (\Omega, \mathbb R^3)\,$ consists of mappings $\,h = (h^1, h^2, h^3) : \Omega \rightarrow \mathbb R^3\,$ whose coordinate functions $\,h^1, h^2, h^3 \,$ belong to $\,\mathscr W^{1,p}(\Omega)\,$ and $\,h(x) \in \mathscr Y\,$ for almost every $\,x \in \Omega\,$. It may be worth reminding the reader that for $\,1\leqslant p < 2=\dim \Omega \,$ topology can inhibit the space $\,\mathscr C^1(\Omega, \mathscr Y)\,$ from being dense in $\,\mathscr W^{1,p} (\Omega, \mathscr Y)\,$. The interested reader is referred to ~\cite{Be, HL, HL2, HIMO} for this issue. This problem, however, disappears completely since our mappings in question are continuous. De facto, when $\, 1 < p < 2\,$, the continuity assumption of monotone Sobolev mappings is critical for the subsequent arguments; it is superfluous, however, when $\,p \geqslant 2\,$.\\
\textit{Royden $\,p$-Algebra}. The class  $\,\mathscr R^p(\Omega) \deff \mathscr C(\overline{\Omega})  \cap \,\mathscr W^{1,p}(\Omega)\,$ consist of real-valued functions in the Sobolev space $\,\mathscr W^{1,p}(\Omega)\,$ that are continuous on $\,\overline{\Omega}\,$. This is a Banach algebra with respect to the
 sub-multiplicative norm
 $$\,|\!|\!| \phi |\!|\!|_{\mathscr R^p(\Omega)} \;\deff \; |\!|\phi |\!|_{\mathscr C(\overline{\Omega})} \;+\;  |\!|\phi |\!|_{\mathscr W^{1,p}(\Omega)}\,, \;\;\;\;\;\; \,|\!|\!| \phi \cdot \psi |\!|\!|_{\mathscr R^p(\Omega)}  \leqslant |\!|\!| \phi |\!|\!|_{\mathscr R^p(\Omega)} \cdot  |\!|\!| \psi |\!|\!|_{\mathscr R^p(\Omega)}\,$$
\subsection{$p$-Harmonic Boundary-Value Problem} We record  less familiar aspects of the Dirichlet problem for the $\,p$-harmonic equation.
\begin{equation}\label{pHarmonicEquation}
\textnormal{div}\, |\nabla \phi |^{p-2} \nabla \phi = 0 \;\;,\;\;\textnormal{ for } \;\phi \in \mathscr W^{1,p}_{\textnormal{loc}}(\Omega) \;\;,\;\; 1 < p < \infty
\end{equation}
 in planar simply connected domains. In general, simply connected domains may have rather odd boundary.

\begin{lemma}\label{DirichletProblem}
Let $\,\Omega \subset \mathbb R^2\,$ be bounded simply connected domain and $\,\Phi \in \mathscr C (\overline{\Omega})\,$. Then there exists unique $\,\phi \in \mathscr C(\overline{\Omega})\, $ that is $\,p$-harmonic in $\,\Omega\, $,  $1< p < \infty$, and agrees with $\Phi\,$ on $\,\partial \Omega\,$.  If $\,\Phi \in \mathscr R^p(\Omega)\,$, then also $\,\phi \in \mathscr R^p(\Omega)\,$. Moreover,
$$\phi \in \Phi \,+\, \mathscr W^{1,p}_\circ(\Omega)\;\;\; \textnormal{and} \;\;\int_\Omega |\nabla \phi |^p \leqslant \int_\Omega |\nabla \Phi |^p
$$
Equality occurs if and only if $\,\Phi = \phi\,$.
\end{lemma}
\begin{proof}
We will only briefly outline the key points of the proof. For a thorough treatment of the Dirichlet problem we refer the reader to  \cite{HKMb} and ~\cite{MZb}.
By virtue of the  {\it Wiener's criterion} the first statement of the lemma holds whenever the complement $\,\mathbb R^2 \setminus \Omega\,$ is \textit{$\,p$-thick} at every boundary point, see ~\cite[Corollary 6.22]{MZb} and ~\cite[(2.22)]{MZb} for a formulation of Wiener's criterion. Simply connected domains indeed satisfy this criterion, a fact not difficult to verify, but it is not explicitly exemplified in the vast literature.
For the second statement we minimize the $\,p$-harmonic energy in the class  $\,\Phi \,+\, \mathscr W^{1,p}_\circ(\Omega)\,$ (so-called variational formulation) which does not require any regularity assumption on the domain $\,\Omega\,$.
 The only point is to show that the variational solution extends continuously to the boundary and it coincides with $\,\Phi\,$. This again follows by Wiener's criterion, see~\cite[Theorem 6.27]{HKMb}.
\end{proof}
\subsection{$p$-Harmonic Replacements}
 Let a map $\,F = u \,+ \,i v : \Omega \rightarrow \mathbb C\simeq \mathbb R^2\,$, defined in a domain $\,\Omega \subset \mathbb C\,$, belong to the Sobolev space $\,\mathscr W^{1,p}_{\textnormal{loc}}(\Omega , \mathbb C) \;,\; 1 < p < \infty\,$.
  \begin{definition}
    $F\,$ is said to be \textit{$\,p$-harmonic} (\textit{coordinate-wise}) if both coordinate functions $\,u\,$ and $\,v\,$ satisfy the equation (\ref{pHarmonicEquation}). We shall introduce the \textit{unisotropic} $\,p$ -\textit{harmonic energy} of $\,F\,$:
    \begin{equation}\,\mathscr E[F] = \mathscr E_{_\Omega}[F] \,\bydef \,\int_\Omega \,\Big( |\nabla u |^p + | \nabla v |^p \Big)\,
    \end{equation}
  \end{definition}$\,$\\
and repeatedly abbreviate this as \textit{energy of $\,F\,$}; because the exponent $\,p\,$ will remain fixed. \begin{remark}
This terminology is different from what can be found in the literature; the term \textit{$\,p$-harmonic mapping} is usually reserved for the coupled $\,p$-harmonic system $\,\textnormal{div} |Df|^{p-2} Df = 0\,$. There is a subtle distinction between these two concepts.
\end{remark}
 Recall from Lemma \ref{DirichletProblem} that for any $\,F \in \mathscr R^p(\Omega, \mathbb R^2)\,$ in a  bounded simply connected domain its boundary map   $\,F : \partial \Omega \rightarrow \mathbb R^2\,$ admits unique  $\,p$-harmonic extension to $\,\Omega\,$. The question arises whether such an extension is injective. The answer depends on how $\,F\,$ runs along $\, \partial \Omega\,$. Of course,  the  boundary map $\,F : \partial \Omega \rightarrow \mathbb R^2\,$ must admit at least one homeomorphic extension inside $\,\Omega\,$. This is also sufficient if $\,F(\partial \Omega) \,$ is a convex curve.
\begin{lemma}\label{pHarmonicReplacement}
\textsl{}Let $\,\Omega \subset \mathbb R^2\,$ be bounded \underline{simply connected} domain and let $\,F \in \mathscr R^p(\Omega, \mathbb R^2)\,$ be $\,p$- harmonic (cordinate wise) in $\,\Omega\,$. Suppose there is  $\,\Psi \in \mathscr C(\overline{\Omega}, \mathbb R^2)\,$ that agrees with $\,F\,$ on $\,\partial \Omega\,$ and takes $\,\Omega\,$ homeomorphically onto a convex domain  $\,\Delta\,$. Then $\,F\,$ is a $\,\mathscr C^\infty$-diffeomorphism of $\,\Omega\,$ onto $\,\Delta\,$.
Moreover, \begin{equation} \,\mathscr E_\Omega [F] \leqslant \mathscr E_\Omega[\Phi]\;, \;\;\textnormal{whenever}\;\; \,\Phi \in F + \mathscr W^{1,p}_\circ(\Omega, \mathbb C)\,.
\end{equation}
\end{lemma}
In this setting  $\,\Psi\,$ plays  the role of the classical Dirichlet boundary data. It helps to capitalize on the topological properties of $\,F\,$  near the (rather weird) boundary $\,\partial \Omega\,$.  We emphasize that $\,\Psi\,$ is not required to have any  Sobolev regularity in $\,\Omega\,$.
 \begin{remark}
Lemma \ref{pHarmonicReplacement} is a $\,p$-harmonic  analogue of the celebrated  Rad\'{o}-Kneser-Choquet Theorem~\cite{Dub}. Under additional assumptions on the domain $\,\Omega\,$, a $\,p$ -harmonic (coordinate wise) analogue of  Rad\'{o}-Kneser-Choquet Theorem was first shown  by Alessandrini and Sigalotti \cite{AS}, see \cite {IKoO} for the isotropic case. The paper  \cite{AS} is concerned with the domains which satisfy the external con condition. This would be redundant, as mentioned in "Remark 3.2" of this paper. However, the essential shortcoming is that neither  formal statement nor the proof in case of non-Jordan domains and non homeomorphic boundary data are provided in \cite{AS}. Thus we must work out additional arguments.
\end{remark}
\begin{proof}
We may assume, in addition to the hypotheses above,  that $\,\Psi\,$ takes $\,\Omega\,$ diffeomorphically onto  $\,\Delta\,$. This is permissible by a theorem of Rad\'o~\cite{Ra}, see also~\cite{Mob}, which asserts that  to every homeomorphism $\,\Psi : \Omega \rightarrow \mathbb R^2\,$ and continuous function $\,\tau : \Omega \rightarrow (0,\infty)\,$ there corresponds a diffeomorphism $\,\Psi_\tau : \Omega \rightarrow \mathbb R^2\,$ such that $\,|\Psi(x) - \Psi_\tau(x) \,| < \tau(x)\,$ in $\,\Omega\,$. Therefore, one can replace $\,\Psi\,$ in Lemma~\ref{pHarmonicReplacement}  by a diffeomorphism $\,\Psi_\tau\,$ with  $\,\tau(x) = \textnormal{dist}(x , \partial \Omega)\,$.\\
Now, consider an increasing sequence of smooth convex domains $\, \Delta_1 \Subset \Delta_2 \Subset \,...\,\Subset \Delta_n \Subset \Delta_{n+1} \,....\, \Subset \Delta$  whose union is $\,\Delta\,$. The corresponding preimages under $\,\Psi\,$ are smooth  Jordan domains in $\,\Omega\,$,
$$\,\Omega_n \deff \Psi^{-1}(\Delta_n)\,,
\;\;\, \Omega_1 \Subset \Omega_2 \Subset \,...\, \Subset\Omega_n \Subset \Omega_{n+1} \,....\,,\;\;\;\;\,\bigcup \Omega_n = \Omega\,.$$
Next, we replace each diffeomorphism $\,\Psi : \overline{\Omega_n\!\!} \onto \overline{\Delta_n\!\!} \,$ \;by a homeomorphism $\,\Psi_n : \overline{\Omega_n\!\!} \onto \overline{\Delta_n\!\!} \,$ which is $\,p$-harmonic (coordinate-wise) in $\,\Omega_n\,$ (thus a diffeomorphism in $\,\Omega_n\,$) and agrees with $\,\Psi\,$ on $\,\partial \Omega_n\,$, see  Theorem 5.1 in \cite{AS}.
We just constructed a sequence of continuous mappings $\,F_n : \overline{\Omega} \onto \overline{\Delta} \,$ which are homeomorphisms on $\,\Omega \,$ and coincide with $\,F\,$ on $\,\partial \Omega\,$,
\begin{displaymath}
F_n= \left \{\begin{array}{ll}
\Psi  & \textrm{in}\,\, \overline{\Omega} \setminus \Omega_n  \\
\Psi_n & \textrm{in}\; \Omega_n \\
\end{array} \right. \;\textnormal{,\;\;with the Jacobian determinants} \;J(x, \Psi_n)  > 0 \;\;\textnormal{in} \,\, \Omega_n
\end{displaymath}
This sequence converges uniformly to $\,F\,$. Indeed, for every $\, x\in \overline{\Omega}\,$ we have
\begin{equation}\label{uniformConvergence}
|F_n(x)  - F(x) | \;\leqslant\; \sqrt{2} \;\big{|}\!\big{|} \Psi  - F   \big{|}\!\big{|}_{\mathscr C(\overline{\Omega} \setminus \Omega_n)} \; \longrightarrow 0
\end{equation}
The latter estimate is trivial when $\,x \in \overline{\Omega} \setminus \Omega_n\,$. To see that this also holds for $\,x\in \Omega_n\,$ we argue by a comparison principle  ~\cite[Comparison principle 7.6, page 133]{HKMb} in a straightforward way. Namely, the coordinate functions of $\,F = (u,v)\,$ and $\,F_n = (u_n, v_n)\,$, being $\,p$-harmonic in $\,\Omega_n\,$,
satisfy
$$
|u_n(x) - u(x) | \leqslant \big{|}\!\big{|} u_n  - u  \big{|}\!\big{|}_{\mathscr C(\partial \Omega_n)} \;\;\;\;\textnormal{and}\;\; \;\;  |v_n(x) - v(x) | \leqslant \big{|}\!\big{|} v_n  - v  \big{|}\!\big{|}_{\mathscr C(\partial \Omega_n)}
$$
which yields the desired estimate in (\ref{uniformConvergence}).
Concerning the positive sign of the Jacobian determinant of $\,F\,$, we shall appeal to a $\,p$-harmonic variant of Hurwitz Theorem~\cite[Theorem 4.9]{IOw=s}; that is,

\begin{theorem}
If a sequence $\,F_n  : \Omega \rightarrow \mathbb R^2\,$ of $\,p$-harmonic (coordinate-wise) orientation preserving diffeomorphisms converges uniformly to $\,F  : \Omega \rightarrow \mathbb R^2\,$ in a domain $\,\Omega \subset \mathbb R^2\,$ ,  then either $\,J(x, F) > 0\,$ everywhere in $\,\Omega\,$ or  $\,J(x, F) \equiv 0\,$ in $\,\Omega\,$.
\end{theorem}

Let us first exclude a  possibility that  $\,J(x, F) \equiv 0\,$ in $\,\Omega\,$. For this, we choose and fix a nonnegative test function $\,\eta \in \mathscr C^\infty_\circ(\Delta)\,$ whose integral mean equals $\,1\,$.  Let $\,\mathbb G \Subset \Delta\,$ denote the support of $\,\eta\,$. Since $\, F_n \rightrightarrows F\,$ (uniformly) and $\, F(\partial \Omega)  = \partial \Delta\,$ it follows that
$$
\bigcup_{n\geqslant 1} F_n^{-1}(\mathbb G) \; \Subset \Omega\;. \;\;\textnormal{In particular,} \;\;\bigcup_{n\geqslant 1} F_n^{-1}(\mathbb G) \;\; \; \Subset \Omega_k \Subset \Omega\,,\;\;\textnormal{for sufficiently large} \, k
$$
Hence, for $\,n> k\,$,  we have
$$
\int _{\Omega_k} \eta(F_n(x)) J(x,F_n) \,\textnormal{d} x = \int_{F_n(\Omega_k)} \eta(y)\,\textnormal{d} y \,\geqslant  \int_{\mathbb G} \eta(y)\,\textnormal{d} y \; = 1\;.
$$
Recall that $\, F_n \rightrightarrows F\,$ on $\,\Omega_{k+1}\Supset \Omega_k\,$. Since $\, F_n\,$ are $\,p$-harmonic in $\,\Omega_{k+1}\,$, we also have $\, DF_n \rightrightarrows DF\,$ on $\,\Omega_{k}\,$. This follows from  local  $\,\mathscr C^{1,\alpha}$-estimates of $\,p$-harmonic functions~\cite{IMa, Uh, Ur}. It is now legitimate to pass to the limit in the above estimate to obtain $ \int _{\Omega_k} \eta(F(x)) J(x,F) \,\textnormal{d} x \geqslant 1\,$. In particular, $\, J(x, F) \not\equiv 0\,$ in $\,\Omega_k\,$. Thus $\,J(x, F) > 0\,$ everywhere in $\,\Omega_k\,$. But $\,k\,$ can be as large as we wish, so $\,J(x, F) > 0\,$ everywhere in $\,\Omega\,$. \\
In particular, $\,F\,$ is a local diffeomorphism in $\,\Omega\,$. On the other hand the  map $\, F : \Omega \onto \Delta\,$ is a uniform limit of homeomorphisms. Therefore, $\,F\,$ is a global diffeomorphism, completing the proof of Lemma \ref{pHarmonicReplacement}.
\end{proof}
\section{Approximation inside a given pre-cell}
 From now on $\,\mathbb Y\,$ is a Lipschitz domain.
 This means that every point $\,y_\circ \in \partial \mathbb Y\,$ has a neighborhood $\,\mathcal O \subset \mathscr Y\,$ and a local ($\,\mathscr C^1$ -smooth) chart $\,\kappa : \mathcal O \rightarrow \mathbb R^2\,$ such that $\,\kappa (\mathcal O \cap \partial \mathbb Y)\,$ is a graph of a Lipschitz function. Regarding the given map $ h \in \mathscr C(\overline{\mathbb X}, \overline{\mathbb Y}) \cap \mathscr W^{1,p}(\mathbb X, \mathbb R^3)\;,\; 1< p < \infty $, we recall that it extends as a monotone map $\,h : \mathscr X \onto \mathscr Y\,$. No regularity outside $\,\overline{\mathbb X}\,$ is required, as this  extension will assist us only in the topological aspects of the proof.\\

As a  preliminary  step in the  proof of Theorem \ref{maintheorem} we shall  approximate  $\,h : \overline{\mathbb X} \onto \overline{\mathbb Y}\,$ with Sobolev mappings which are univalent  within a given pre-cell, and  remain unchanged outside the pre-cell. The approximation is understood by means of the metric in the \textit{Royden space} $\,\mathscr R^{\;\!p}(\mathbb X, \mathbb Y) = \mathscr C(\overline{\mathbb X}, \overline{\mathbb Y}) \cap \mathscr W^{1,p}(\mathbb X, \mathbb R^3)\,$; that is, uniformly and strongly in $\,\mathscr W^{1,p}(\mathbb X, \mathbb R^3)\,$. Precisely, we have:

 \begin{proposition}\label{ApproximationInPrecell}
 Let $\,\mathcal Q \subset \mathbb Y\,$ be a cell in $\,\mathbb Y\,$, which we assume to be Lipschitz regular, and $\,\mathcal U\,$  be its pre-cell in $\mathbb X\,$. Then there exists a sequence of monotone mappings $\,h_j : \overline{\mathbb X} \onto \overline{\mathbb Y}\,$ such that:
  \begin{itemize}
  \item[(a)] $\,h_j \in  \mathscr C(\overline{\mathbb X}, \overline{\mathbb Y}) \cap \mathscr W^{1,p}(\mathbb X, \mathbb R^3)\,,\;\;\; j = 1,2, ...\;,$
      \item[(b)] $\,h_j = h \,\;\textnormal{on} \;\; \overline{\mathbb X} \setminus \mathcal U\,,$
      \item[(c)] $\, h_j : \mathcal U \onto \mathcal Q\,\; \textnormal{are homeomorphisms} ,$
      \item[(d)] $\, h_j \rightrightarrows h\,\, \textnormal{uniformly in }\, \overline{\mathbb X}\,,$
      \item[(e)] $\,h_j \rightarrow h\,\, \textnormal{strongly in } \, \;\mathscr W^{1,p}(\mathbb X , \mathbb R^3)\,,$
  \end{itemize}
  \end{proposition}
 \begin{remark}Upon the extension $\,h_j \bydef h : \mathscr X \setminus \mathbb X \onto \mathscr Y \setminus \mathbb Y\,$, we obtain monotone mappings between reference manifolds, again denoted by $\,h_j : \mathscr X \onto \mathscr Y\,$.
 \end{remark}

 \begin{proof} It will takes 5 steps to complete the proof of Proposition \ref{ApproximationInPrecell}.

\textbf{\textit{Step 1}}. \textit{Reduction to simply connected Jordan domains in $\,\mathbb S^2\,$}.  We  view $\,\mathbb S^2\,$ as the extended complex plane $\,\widehat{\mathbb C} = \mathbb C \cup\{\infty\} \simeq \widehat{\mathbb R^2}\,$ equipped with the chordal metric.
Choose and fix a simply connected neighborhood $\, \mathcal Q\,' \Supset \overline{\mathcal Q}\,$ such that $\,\mathcal Q' \cap \partial \mathbb Y\,$ becomes a Lipschitz regular Jordan arc. Obviously  such $\mathcal Q'$ does exist. Thus $\, \mathcal Q\,'\cap \mathbb Y\,$ is a boundary cell in $\,\mathbb Y\,$, even if $\,\mathcal Q\,$ was an internal cell. Denote by  $\,\mathcal U\,' = h^{-1}(\mathcal Q \,') \subset \mathscr X\,$. Since $\,h\,$ is monotone,  $\,\mathcal U\, ' \,$ is a simply connected neighborhood of the continuum $\,h^{-1}(\overline{\mathcal Q}) \Subset \mathcal U\,'\,$ and $\, h :  \partial \mathcal U\,' \rightarrow \partial \mathcal Q\,'\,$. Take a look at the
commutative diagram
\begin{eqnarray}
\mathcal U\,' \!\!\!\!& -\!\!\!-\!\!\!-\!\!\!- h -\!\!\!-\!\!\!\!\rightarrow &\!\!\!\!\mathcal Q\,'\nonumber\\
\,\mid          &\;\;\;\;\;\;\;& \!\!\mid\nonumber \\
   \phi &\;\;\;\;\;\;\;\;\; & \!\!\!\psi \;\;\nonumber\\
   \downarrow & \;\;\;\;\;\;\;\; & \!\!\downarrow \;\;\nonumber\\
   \mathbb R^2 \!\!\!\!&  -\!\!\!-\!\!\!-\!\!\!- \widetilde{h} -\!\!\!-\!\!\!\!\rightarrow & \!\!\!\!\mathbb R^2 \nonumber
\end{eqnarray}
where $\,\phi : \mathcal U\,' \onto \mathbb R^2\,$ and $\,\psi : \mathcal Q\,' \onto \mathbb R^2\,$ are $\,\mathscr C^1\,$-diffeomorphisms. For the existence of such diffeomorphisms one may appeal to the uniformization theorem \cite{Poincare} and \cite{Koebe1, Koebe2, Koebe3}. It asserts that every simply connected Riemann surface is conformally equivalent to either open unit disk, the complex plane or the standard sphere; thus in our case, $\,\mathscr C^1$ -diffeomorphic to $\,\mathbb R^2\,$. Note that $\,\widetilde{h} = \psi\circ h\circ \phi^{-1}  :\mathbb R^2 \onto \mathbb R^2\,$ extends continuously  to the one-point compactification of the plane; that is, to a monotone mapping of the Riemann sphere onto itself, still  denoted by $\,\widetilde{h} : \widehat{\mathbb C} \onto \widehat{\mathbb C}\;,\; \widetilde{h}(\infty) = \infty\,$. With the aid of stereographic projections we  move the sets $\,\mathbb X \cap \varphi (\mathcal U\,')\,$ and $\,\mathbb Y\cap \psi(\mathcal Q\,')\,$ away from $\,\infty\,$.
Now the proof of Proposition \ref{ApproximationInPrecell} reduces to the case in which
 \begin{eqnarray}
 &&\lefteqn{\mathscr X =  \widehat{\mathbb C} \;, \;\mathscr Y = \widehat{\mathbb C} \;}\nonumber \\
  &&\mathbb X , \mathbb Y \subset \mathbb C\,\;\textnormal{are bounded simply connected Jordan domains,} \;\nonumber\\
    && \mathbb Y\, \textnormal{and the cell} \; \mathcal Q \subset \mathbb Y \,\textnormal{are Lipschitz domains.}\nonumber
 \end{eqnarray}

 \textbf{\textit{Step 2}}. \textit{Further reduction}. We shall need $\,\mathcal Q \subset \mathbb Y\,$  to be convex; for instance, the unit square. For this, since  $\,\mathcal Q \subset \mathbb R^2\,$ is  Lipschitz domain, one might try to use a bi-Lipschitz transformation $\,F : \mathbb R^2 \onto \mathbb R^2\,$, such that $\, F : \mathcal Q  \onto (0, 1) \times(0, 1)\,$.

\begin{remark} Although the existence of such $\,F\,$ poses no problem, a caution must be exercised.  Every bi-Lipschitz map  $\,F : \mathbb A \onto \mathbb B\,$ between bounded planar domains and its inverse $\,F^{-1} : \mathbb B \onto \mathbb A\,$ induce  bounded (nonlinear) composition operators:
$$
F_\sharp  : \,\mathscr W^{1,p}(\Omega, \mathbb A ) \,\rightarrow \,\mathscr W^{1,p}(\Omega, \mathbb B )\;,\;\;\textnormal{by the rule}\;\; F_\sharp(g) \deff F\circ g \;,
$$
$$
F^{-1}_\sharp  : \,\mathscr W^{1,p}(\Omega, \mathbb B ) \,\rightarrow \,\mathscr W^{1,p}(\Omega, \mathbb A )\;,\;\; \;\textnormal{by the rule}\;\; F^{-1}_\sharp(f) \deff F^{-1}\circ f \;,
$$
whatever the domain $\,\Omega \subset \mathbb R^2\,$ is. But in general the continuity of these operators is  questionable \cite{Ha1}. Fortunately, there is a satisfactory solution to this puzzle. For Lipschitz domains such as $\, \mathbb A =\mathcal Q\,$ and $\, \mathbb B = (0, 1) \times (0, 1)\,$  one can construct special  bi-Lipschitz transformation $\,F : \mathbb R^2 \onto \mathbb R^2\,$,  $\,F : \mathbb A\onto\mathbb B\,$, for  which the induced composition operators $\,F_\sharp\,$ and $\, F^{-1}_\sharp\,$ are indeed continuous. Actual construction of such $\,F\,$ is presented in \cite{IOw=s}.
\end{remark}
Thus we may assume that $\,\mathcal Q = (0, 1) \times (0, 1) \subset \mathbb Y \,$. Furthermore, in case of a boundary cell, we  assume that its external face $\,\overline{\mathcal Q} \cap \partial \mathbb Y\,$ equals $\, \{(x, 0)\,,   \;  0 \leqslant x \leqslant 1\,\}\,$.\\

\textbf{\textit{Step 3}}. \textit{Covering by small cells}. Given $\,\varepsilon >0\,$, we cover $\,\mathcal Q\,$ by the family of overlapping open squares $\,\mathcal Q_1 , \mathcal Q_2, ... , \mathcal Q_N\,$ of diameter less then $\,\varepsilon\,$ in which every point in $\,\mathcal Q\,$ belongs to at most three of the closed squares in this family. In mathematical terms,\\

\begin{itemize}
\item $\,\mathcal Q = \,\mathcal Q_1 \cup \mathcal Q_2 \cup ... \cup \mathcal Q_N\,$,\;\;\;\;\;$\; N = N(\varepsilon)\,$\vskip0.1cm
\item  $\,\textnormal{diam} \,\mathcal Q_i \; < \varepsilon\;,\;\;\;\;\;\;\;\;\;\;\;\;\;\;\;\;\;\;\;\;\;\;\;\; i = 1,2,..., N \,$\vskip0.1cm
\item  $\, 1 \leqslant \sum_{i=1}^N \textnormal{\Large $\chi$}_{\overline{\mathcal Q}_{\,i}}(z)\; \leqslant 3\;, \;\;\;\;\;\;\;\;\;\textnormal{for} \;z \in \mathcal Q \,$\vskip0.1cm
\item Each $\,\mathcal Q_i\,$ is either internal or a boundary cell for $\,\mathbb Y\,$.
\end{itemize}
Construction of such a cover poses no difficulty. We now consider the corresponding pre-cells in $\, \mathcal  U_i = h^{-1} (\mathcal Q_i) \subset \mathcal U \subset \mathbb X\,$.
\begin{itemize}
\item $\,\mathcal U = \,\mathcal U_1 \cup \mathcal U_2 \cup ... \cup \mathcal U_N\,$,\;\;\;\;\;$\; \;\,\;N = N(\varepsilon)\,$\vskip0.1cm
\item  $\, 1 \leqslant \sum_{i=1}^N \textnormal{\Large $\chi$}_{_{\mathcal U_i}}(x)\; \leqslant 3\;, \;\;\;\;\;\;\;\;\;\textnormal{for} \;x \in \mathcal U \,$\vskip0.1cm
    \item Each $\,\mathcal U_{\,i}\,$ is either internal or a boundary pre-cell in $\,\mathbb X\,$.
\end{itemize}
It may be worth pointing out that the pre-cells $\,\mathcal U_i\,$ can remain very large in diameter as $\,\varepsilon\,$ approaches zero. Typically this occurs when  a continuum collapses into a point.\\

\textbf{\textit{Step 4}}. \textit{A chain of $\,p$-harmonic replacements}. We are going to construct by induction a chain $\,f_1\rightsquigarrow f_2\rightsquigarrow... \rightsquigarrow f_N \rightsquigarrow f_{N+1}\,$  of  mappings $\, f_i : \widehat{\mathbb C} \onto \widehat{\mathbb C}\,$. We set  $\,f_1 = h\,$. For the induction step suppose we are given a monotone map $\,f_i  : \overline{\mathbb X} \onto \overline{\mathbb Y}\,$ of Sobolev class $\,\mathscr W^{1,p} (\mathbb X, \mathbb Y)\,$  and its monotone extension $\,f_i  : \widehat{\mathbb C} \onto \widehat{\mathbb C}\,$. Consider the cell $\,\mathcal Q_i \subset \mathcal Q \subset \mathbb Y\,$ (internal or boundary) and the corresponding  pre-cell $\,\mathcal U_{\,i} \subset \mathcal U \subset \mathbb X\,$ for the mapping  $\,f_i\,$. By Lemma \ref{Homeomorphic replacement}, there exists a monotone map $\,\Psi_i  : \overline{\mathbb C} \onto \overline{\mathbb C} \,$ such that $\, \Psi_i : \mathcal U_{\,i} \onto \mathcal Q_i\,$ is a homeomorphism and $\, \Psi_i \equiv f_i : \widehat{\mathbb C} \setminus \mathcal U_{\,i} \onto \widehat{\mathbb C} \setminus \mathcal Q_i\,$. Then, with the aid of  Lemmas  \ref{DirichletProblem}  and \ref{pHarmonicReplacement}, we can modify $\,f_i\,$ within the pre-cell $\,\mathcal U_{\,i}\,$ to obtain,
\begin{displaymath}
f_{i+1} = \left\{ \begin{array}{ll}
f_i &\textrm{in $\,\widehat{\mathbb C} \setminus \mathcal U_{\,i}$} \\
\textrm{{$\,p\,$}-harmonic of class $\,f_{i} + \mathscr W^{1,p}_\circ(\mathcal U_{\,i})\,$} & \textrm{in  $\,\mathcal U_{\,i}\,$}
\end{array} \right.
\end{displaymath}
 We emphasize (in view of Lemmas \ref{Homeomorphic replacement} and \ref{pHarmonicReplacement})  that $\, f_{i+1}\,$ takes $\,\mathcal U_{\,i}\,$ homeomorphically onto $\,\mathcal Q_i\,$, whereas under the map $\,f_i\,$ some points in $\,\mathcal U_{\,i}\,$ may collapse into $\,\partial\mathcal Q_i\,$.\\
Here are the essential properties of these mappings. It may be worth reminding the reader that we are dealing with rather weird domains.
\begin{itemize}
\item Each $\,f_i\,$ belongs to the Sobolev class $\,\mathscr W^{1,p}(\mathbb X, \mathbb Y)\,$. This is due to Lemma \ref{DirichletProblem} which yields $\, f_{i+1} - f_i \in \,\mathscr W^{1,p}_\circ(\mathcal U _{\,i}, \mathbb C)\,$.\\
    \item The energies are nonincreasing;  \;\;$\,\mathscr E_\mathbb X [f_i] \leqslant \mathscr E_\mathbb X [f_{i+1}] \leqslant \, ....\; \leqslant \mathscr E_\mathbb X [h] \,$.\\
        \item  Each $\,f_{i+1}\,$ is locally injective in $\,\mathcal U_{\,1} \cup \mathcal U_{\,2} \cup ... \cup \mathcal U_{\,i}\,$ (no branch points). This is because when making the $\,p$ -harmonic replacement of $\,f_i\,$ we gained  new points of local injectivity for $\,f_{i+1}\,$. These are all points in $\,\mathcal U_{\,i}\,$. At the same time we did not loose local injectivity at the points where $\,f_i\,$ was already injective. \\
            \item $\, f_{N+1}  : \mathcal U \onto \mathcal Q\,$ is a homeomorphism, because it is a local homeomorphism  and the preimage of any point is a continuum in $\,\overline{\mathbb X}\,$.\\

                \item  For each $\,z \in \mathcal U\,$, we have $\,|f_{N+1}(z) \, - h(z) \,| \,\leqslant 3\,\varepsilon\,$. Indeed, by triangle inequality,
                    $$ \,|f_{N+1}(z) \, - h(z) \,| \,\leqslant  \sum_{i= 1}^N  |f_{i+1}(z) \, - f_i(z) \,|$$
                    Let us take a quick look at each term $\,|f_{i+1}(z) \, - f_i(z) \,|\,$.  If $\,z \in \mathcal U_{\,i}\,$ then both $\, f_{i+1}\,$ and $\,f_i\,$ lie in $\,\overline{\mathcal Q_i}\,$; therefore, $\,|f_{i+1}(z) \, - f_i(z) \,| \leqslant \textnormal{diam}\,  \mathcal Q_ i \leqslant \varepsilon\,$. This term vanishes if $\,z \not\in \mathcal U_{\,i}\,$. But, for a given point $\,z\,$ there can be at most three pre-cells containing $\,z\,$. In other words the above sum consist of at most three nonzero terms, each of which does not exceed $\,\varepsilon\,$.
                    \end{itemize}
\textbf{\textit{Step 5}}. \textit{Letting $\,\varepsilon \,$ small}\,. We are now ready to proceed to the final construction of the mappings $\,h_j : \overline{\mathbb X} \onto \overline{\mathbb Y}\,$,

\begin{displaymath}
h_j  = \left\{ \begin{array}{ll}
f_{N+1}(z)\,, \; \textnormal{where}\;\;  N=N(\varepsilon)\,,\; \varepsilon =  1/j\,\; & \textrm{if $ \,z \in \mathcal U\,$}\\
$\,$\\
h(z) & \textrm{if $\,z \in \overline{\mathbb X} \setminus \mathcal U\,$}
\end{array} \right.
\end{displaymath}
Obviously, we have $\,| h_j(z) \, - h(z) \,| \,\leqslant  3/j\,$ everywhere in $\,\overline{\mathbb X}\,$. Hence $\, h_j \rightrightarrows h\,$\, uniformly in  $\, \overline{\mathbb X}\,$. To complete the proof of Proposition \ref{ApproximationInPrecell}, we need only verify that   $\,h_j \rightarrow h\,\, \textnormal{strongly in } \, \;\mathscr W^{1,p}(\mathbb X , \mathbb R^2)\,$. The crucial ingredient is that $\,\mathscr E_\mathbb X [h_j] \leqslant \,\mathscr E_\mathbb X [h]\,$, for all $\,j = 1,2, ...\;.$ In particular, $\,h_j \,$ converge to $\, h\,  \textnormal{weakly in } \, \;\mathscr W^{1,p}(\mathbb X , \mathbb R^2)\,$. Now the lower semicontinuity of the energy functional yields  $\,\mathscr E_\mathbb X [h] \leqslant \,\liminf\mathscr E_\mathbb X [h_j] \leqslant \,\liminf\mathscr E_\mathbb X [h] = \mathscr E_\mathbb X [h]\,$. This in turn implies, by uniform convexity arguments, that $\,h_j\,$ converge to $h$ strongly in $\,\mathscr W^{1,p}(\mathbb X, \mathbb R^2)\,$.
\end{proof}
\begin{remark}
Here is a careful look at the uniform convexity arguments. Consider the coordinate functions for $\, h_j = u_j + i\, v_j\,$ and $\,h = u + i\, v\,$. In view of lower semicontinuity, we have  $\,\int |\nabla u |^p \leqslant  \liminf \int |\nabla u_j |^p \,$ and $\,\int |\nabla v |^p \leqslant  \liminf \int |\nabla v_j |^p \,$. Adding these inequalities, we obtain  $\,\int \big(|\nabla u |^p + |\nabla v |^p\big) \leqslant \liminf \int \big(|\nabla u_j |^p + |\nabla v_j |^p\big) = \liminf \mathscr E[h_j]  \leqslant \mathscr E[h] = \int \big(|\nabla u |^p + |\nabla v |^p\big)\,$, which is possible only when $\,\int |\nabla u |^p =  \liminf \int |\nabla u_j |^p \,$ and $\,\int |\nabla v |^p =  \liminf \int |\nabla v_j |^p \,$. Now we see that $\,\nabla u_j\,$ and $\,\nabla v_j\,$ converge strongly in $\,\mathscr L^p(\mathbb X, \mathbb R^2)\,$, because the usual normed space $\,\mathscr L^p(\mathbb X, \mathbb R^2)\,$ is uniformly convex, by Clarkson's Inequality for vectors in the Euclidean space $\,\mathbb R^2\,$, see ~\cite{Cl}.
Actually, one could apply  Theorem 2 in \cite{Day}  to infer  that the Banach space \,$L^p(\mathbb X, \mathbb R^2) \times L^p(\mathbb X, \mathbb R^2)\,$, equipped with the norm $\,\big[\int_\mathbb X \left( |f|^p \,+\,|g|^p \right) \,\big]^{1/p}\,$, is uniformly convex as well.
 \end{remark}
\section{Completing the proof of Theorem \ref{maintheorem}} We now return to the surfaces $\,\mathbb X \subset \mathscr X\,$ and $\,\mathbb Y\subset \mathscr Y\,$ and construct the mappings $\,h_j : \overline{\mathbb X} \onto \overline{\mathbb Y}\,$ stated in Theorem \ref{maintheorem}. The arguments are similar to those used in  \textbf{Steps 3,4}\; and \textbf{5}. The main difference, however,  is that we now choose and fix one particular  finite cover of $\,\mathbb Y\,$ by cells. Having Proposition \ref{ApproximationInPrecell} in hands there will be no need to partition those cell into smaller cells. We adopt  analogous notation. Thus,  we let  $\,\mathbb Y\subset \mathscr Y \,$ be covered by Lipschitz cells $\,\mathcal Q_1 , \mathcal Q_2, ... , \mathcal Q_N\,$, including both internal and boundary cells. This time $\,N\,$ is fixed for the rest of our proof. In symbols,
\begin{itemize}
\item $\,\mathbb Y = \,\mathcal Q_1 \cup \mathcal Q_2 \cup ... \cup \mathcal Q_N\,$ \vskip0.2cm
\item  $\, 1 \leqslant \sum_{i=1}^N \textnormal{\Large $\chi$}_{\overline{\mathcal Q}_{\,i}}(y)\; \leqslant N\;, \;\;\;\;\;\;\;\;\;\textnormal{for} \;y \in \mathbb Y \,$
\end{itemize}
Let $\,\varepsilon\,$ be any positive number. As before, we proceed by induction to define a chain $\,F_1\rightsquigarrow F_2\rightsquigarrow... \rightsquigarrow F_N \rightsquigarrow F_{N+1}\,$  of monotone mappings $\, F_i : \overline{\mathbb X } \onto \overline{\mathbb Y}\,$. The first map is   $\,F_1 \bydef h\,$.
In the induction step we appeal to Proposition \ref{ApproximationInPrecell}.
 Suppose we are given a monotone map $\,F_i  : \overline{\mathbb X} \onto \overline{\mathbb Y}\,$ of Sobolev class $\,\mathscr W^{1,p} (\mathbb X, \mathbb Y)\,$  and its monotone extension $\,F_i  : \widehat{\mathscr X} \onto \widehat{\mathscr Y}\,$. Consider the cell $\,\mathcal Q_i  \subset \mathbb Y\,$ and the corresponding  pre-cell $\,\mathcal U_{\,i}  \subset \mathbb X\,$ for $\,F_i\,$. Then, by Proposition \ref{ApproximationInPrecell}, there exists a monotone map
  $\,F_{i+1} : \overline{\mathbb X} \onto \overline{\mathbb Y}\,$ such that:
  \begin{itemize}
  \item[(a)] $\,F_{i+1} \in  \mathscr C(\overline{\mathbb X}, \overline{\mathbb Y}) \cap \mathscr W^{1,p}(\mathbb X, \mathbb R^3)\,$\vskip0.1cm
      \item[(b)] $\,F_{i+1} = F_i \; : \; \overline{\mathbb X} \setminus \mathcal U_{\,i} \onto \overline{\mathbb Y} \setminus \mathcal Q_{\,i}\,,$\vskip0.1cm
      \item[(c)] $\, F_{i+1} : \mathcal U_{\,i} \onto \mathcal Q _i\,\; \;\;\;\;\textnormal{is a homeomorphism} ,$\vskip0.1cm
      \item[(d)] $\,| F_{i+1} - F_i | \leqslant  \frac{\varepsilon}{2 N} \, \;\;\;\;\textnormal{everywhere in }\, \overline{\mathbb X}\,,$\vskip0.1cm
      \item[(e)] $\,\|\;DF_{i+1}  -  DF_i \;\|_{\mathscr L^p(\mathbb X)}  \leqslant \frac{\varepsilon}{2 N} $
  \end{itemize}

     In each induction step, passing from $\,F_i\,$ to $\,F_{i+1}\,$ we gain injectivity of $\,F_{i+1}\,$  within  $\,\mathcal U_{\,i}\,$, and at the same time produce no branch points in the previous pre-cells. Thus $\,F_{N+1} : \mathbb X \onto \mathbb Y\,$ is a local homeomorphism. Arguing as in \textbf{Step 4}, since $\,F_{N+1} : \overline{\mathbb X} \onto \overline{\mathbb Y}\,$ is monotone, we see that $\,F_{N+1} : \mathbb X \onto \mathbb Y\,$ is a homeomorphism. Lastly, by triangle inequality, we obtain
     \begin{eqnarray}
    & &\,|\!|\!| F_{N+1} - h  |\!|\!|_{\mathscr R^p(\mathbb X)} \;\deff \; |\!|F_{N+1} - F_1 |\!|_{\mathscr C(\overline{\mathbb X})} \;+\;  |\!|DF_{N+1} - DF_1 |\!|_{\mathscr L^{p}(\mathbb X)} \,\;\leqslant\;\;\;\;\; \nonumber \\
     & & \sum_{i= 1}^N  |\!|F_{i+1} - F_i |\!|_{\mathscr C(\overline{\mathbb X})}\, +\;\sum_{i= 1}^N  |\!|DF_{i+1} - DF_i |\!|_{\mathscr L^{p}(\mathbb X)}  \; \leqslant \frac{\varepsilon}{2} \, +\,\frac{\varepsilon}{2} \;= \varepsilon\nonumber
     \end{eqnarray}
 The proof of Theorem \ref{maintheorem} is concluded by setting $\, \varepsilon = \, 1/j\,$ and $\,h_j = F_{N+1}\,$.

\section{Applications to Thin Plates and Films}\label{Application} Let us  demonstrate  the utility of Theorem \ref{maintheorem}  by establishing the existence of the energy-minimal deformations of thin plates (planar domains) and films (surfaces) for $\,p$ -harmonic type energy.
Recall the reference manifolds and the Jordan domains $\,\mathbb X\subset \mathscr X\,$ and $\,\mathbb Y \subset \mathscr Y\,$.  Here we assume that both $\,\mathbb X\,$ and $\,\mathbb Y\,$  are Lipshitz domains. We examine homeomorphisms $\, h : \mathbb X \rightarrow \mathbb Y\,$ in the Sobolev space $\,\mathscr W^{1,p}(\mathbb X, \mathbb Y) \subset\,\mathscr W^{1,2}(\mathbb X, \mathbb Y)\, , \;2 \leqslant p< \infty\,$,  and their weak limits. Every homeomorphism $\,h \in\,\mathscr W^{1,2}(\mathbb X, \mathbb Y)\,$ extends up to the boundary as a continuous monotone map, still denoted by $\,h : \overline{\mathbb X} \onto \overline{\mathbb Y}\,$. Moreover, we have a uniform bound of the modulus of continuity in terms of the Dirichlet and the $\,\mathscr E_p\,$- energy on a surface, see Subsection \ref {EnergyOnSurface} below for the definition of $\,\mathscr E_p\,$.
\begin{equation}\label{UniformContinuity}
|h(x_1) - h(x_2) |^2 \;\leqslant  \frac{ C(X,Y)\;\mathscr E_2[h]}{\log \big( e + \frac{1}{|x_1 - x_2 |} \,\big) } \;\leqslant  \frac{ C_p(X,Y)\;\big [\,\mathscr E_p[h]\,\big]^{2/p}}{\log \big( e + \frac{1}{|x_1 - x_2 |} \,\big) }
\end{equation}
For $\,h\,$ fixed, its logarithmic modulus of continuity  was already known by Lebesgue \cite{Lebesgue}. However, it is the dependence on the energy of $\,h\,$ that we are specifically concerned (to apply limiting arguments).  
The proof of estimate \eqref{UniformContinuity} runs along similar lines as for Lipschitz planar domains in~\cite{IO}. There are, however,  routine adjustments necessary to fit the arguments to 2 -dimensional surfaces.\\
 Now a $\,\mathscr W^{1,p}$ -weakly converging sequence of homeomorphisms between $\,\mathbb X\,$ and $\,\mathbb Y\,$ actually converges uniformly. Based on Theorem \ref{maintheorem}\,,  we have
\begin{equation}
\overline{\mathscr H_p}(\mathbb X, \mathbb Y) =  \widetilde{\mathscr H_p}(\mathbb X, \mathbb Y) = \mathscr M_p(\mathbb X, \mathbb Y)
\end{equation}
meaning that (respectively) the strong closure, the sequential weak closure, and the monotone maps in the Sobolev space $\,\mathscr W^{1,p}(\mathbb X, \mathbb Y)\,$ are the same thing.
\subsection{The isotropic $\,p$ -harmonic integral on surfaces}\label{EnergyOnSurface} An intrinsic example of the energy-minimal deformations of thin plates and films is furnished by the $\,p$ -harmonic integral, $\,2\leqslant p < \infty\,$. Suppose we are given a monotone Sobolev map  $\,h \in \mathscr M_p(\mathbb X, \mathbb Y)\,$. To almost every point $\,x \in \mathbb X\,$ there corresponds the linear tangent
map $\, Dh(x)  : \mathbf T_x(\mathbb X)  \into \mathbf T_y(\mathscr Y)\,, \,y = h(x) \in \overline{\mathbb Y} \,$,  and its adjoint  $\, D^*h(x)  : \mathbf T_y(\mathscr Y)  \into \mathbf T_x(\mathbb X) \,$ with respect  to the scalar products in $\, \mathbf T_x(\mathbb X)\,$ and $\,\mathbf T_y(\mathscr Y)\,$. The Cauchy-Green stress tensor $\, \mathbf G_ h \bydef [D^*h]\circ[Dh]  : \mathbf T_x(\mathbb X) \rightarrow \mathbf T_x(\mathbb X)\,$ gives rise to the Hilbert-Schmidt norm of the tangent map, $\, |Dh|  = [\,\textnormal{Trace} \, \mathbf G_ h\,]^{1/2}\,$. Now the isotropic  $p$ -harmonic energy of $\,h\,$  is defined by:
\begin{equation}\label{pHarmonicEnergy}
\mathscr E_p\,[h]  \bydef \int_\mathbb X |Dh(x)|^p\,\textnormal {d} x \;,
\end{equation}
where the area element $\,\textnormal{d} x\,$  is the one induced by the Riemannian metric in $\;\mathbb X\,$. The term \textit{isotropic} refers to the fact that the integrand is invariant under the rotations in $\,\mathbf T_x(\mathbb X)\,$ and $\,\mathbf T_y(\mathbb Y)\,$.
We call $\mathscr E_2 [h] \,$ the Dirichlet energy. The energy in~\eqref{pHarmonicEnergy} fits to the following more general scheme:
\begin{equation}\label{GeneralEnergy}
\mathscr E\,[h]  \bydef \int_\mathbb X \mathbf E (x, h, Dh) \,\textnormal {d} x \;,\;\;\;\textnormal{for mappings}\; \;h \in \mathscr M_p(\mathbb X, \mathbb Y)\,,
\end{equation}
where $\,\mathbf E (x, y, L)\,$ is a given real-valued function defined for $\,x \in \mathscr X\,$ and  $\,y \in \mathscr Y\,$,  and the linear maps $\,L : \mathbf T_x \mathscr X\,\into\,\mathbf T_y \mathscr Y\,\,$.  We shall impose the following conditions on the energy integral  in (\ref{GeneralEnergy});  they  suffice for the application of the Direct Method in the Calculus of Variations:
\begin{itemize}
   \item \textit{coercivity};  \\
    $\,c\int_\mathbb X |Dh(x)|^p\,\textnormal {d} x  \preccurlyeq \int_\mathbb X \mathbf E (x, h, Dh) \,\textnormal {d} x\, \preccurlyeq C\int_\mathbb X |Dh(x)|^p\,\textnormal {d} x \,$.
    \vskip0.2cm
    \item \textit{continuity in the strong topology of $\,\mathscr W^{1,p}(\mathbb X, \mathbb Y)\,$; }\\
    $\, \mathscr E\,[h] = \liminf \mathscr E\,[h_j]\;,\;\;\textnormal{whenever}\; h_j \in \mathscr M_p(\mathbb X, \mathbb Y) \,\,\textnormal{converge strongly to} \;h\,$.\vskip0.2cm
    \item \textit{lower semicontinuity};\\ $\, \mathscr E\,[h] \leqslant \liminf \mathscr E\,[h_j]\;,\;\;\textnormal{whenever}\; h_j \in \mathscr M_p(\mathbb X, \mathbb Y) \,\,\textnormal{converge weakly to} \;h\,$.
\end{itemize}\vskip0.1cm
These conditions hold, in particular, for the $\,p$ -harmonic integral in (\ref{pHarmonicEnergy}). Next, choose and fix a homemorphism $\,\varphi \in \mathscr H_p(\mathbb X, \mathbb Y)\,$ and a compact set $\,\Gamma \subset \partial \mathbb X\,$ (empty in case of traction free problems). We consider the class of Sobolev homeomorphisms,
\begin{equation}
 \mathscr H_p(\mathbb X, \mathbb Y , \Gamma ; \varphi)  \bydef \{ h\in \mathscr H_p(\mathbb X, \mathbb Y) ; \;h_{|\Gamma} = \varphi_{|\Gamma}\,\, \; (\textnormal{upon extension to}\; \partial \mathbb X)\,\}
\end{equation}
By Theorem \ref{maintheorem}, \,its strong and sequential weak closures in  $\,\mathscr W^{1,p}(\mathbb X, \mathbb R^2)\,$ are the same and coincide with the monotone mappings of the class:

\begin{equation}
 \mathscr M_p(\mathbb X, \mathbb Y , \Gamma ; \varphi)  \bydef \{ h\in \mathscr M_p(\mathbb X, \mathbb Y) ; \;h_{|\Gamma} = \varphi_{|\Gamma} \; \}
\end{equation}{\,}\\
Straightforward application of the direct method in the Calculus of Variations yields
\begin{theorem} There always exists the energy-minimal map $\,h_\circ \in \mathscr M_p(\mathbb X, \mathbb Y , \Gamma ; \varphi)\,$ such that
\begin{equation}\nonumber
\mathscr E\,[h_\circ] \,=\, \min_{h \in \mathscr M_p(\mathbb X, \mathbb Y , \Gamma ; \,\varphi)}  \mathscr E\,[h] =\, \inf_{h \in \mathscr H_p(\mathbb X, \mathbb Y , \Gamma ; \,\varphi)}  \mathscr E\,[h]\;\;\;\;(\textnormal{no \textit{Lavrentiev phenomenon}})
\end{equation}
\end{theorem}

\begin{remark}
The existence of energy-minimal mappings within $\mathscr M_p(\mathbb X, \mathbb Y )$, $p\ge 2$, can also be obtained for many realistic variational integrals in NE, including neohookean energy;
\[\mathcal E [h]= \int_\mathbb X \left(\abs{Dh}^2 + \frac{1}{\det Dh}\right)\qquad \textnormal{for } \; h\in \mathscr M^+_2(\mathbb X, \mathbb Y)\]
Here $ \mathscr M^+_2$ stands for mappings in  $\mathscr M_2(\mathbb X, \mathbb Y) $ with positive Jacobian determinant. For  $\ell$-connected, $2\le \ell < \infty$, Lipschitz domains $\X$ and $\Y$ the class $\mathscr M_2(\mathbb X, \mathbb Y )$, as oppose to $\mathscr H_2(\mathbb X, \mathbb Y )$, is closed under weak convergence in $\W^{1,2} (\X, \Y)$. Applying the direct method in the calculus of variations we see that there always exists an energy-minimal map $h_\circ \in \mathscr M^+_2(\mathbb X, \mathbb Y )$. Now, by Theorem~\ref{maintheorem} the minimizer $h_\circ$ can be approximated uniformly and strongly in $\W^{1,2} (\X , \R^2) $ by homeomorphisms. The key issue is whether
\begin{equation}\label{kavtaht}
\mathcal E [h_\circ] =  \inf_{h \in \mathscr H_2(\mathbb X, \mathbb Y)} \mathcal E [h] 
\end{equation}
It is interesting to notice that~\eqref{kavtaht} holds if \[\frac{1}{\det Dh_\circ} \in \mathscr L^\infty_{\loc} (\X).\] Indeed, this condition implies that $h_\circ$ has locally integrable distortion; that is,
\[\abs{Dh_\circ (x)}^2 \le K(x)\, \det Dh_\circ (x)\, , \quad \textnormal{ where } K\in \mathscr L^1_{\loc} (\X)\, .\]
By~\cite{IS} $h_\circ$ is open and discrete and, being monotone, is a homeomorphism of class $\mathscr H_2 (\X, \Y)$.

\end{remark}
\section{Afterward}
 Remarkably, the existence of traction free minimal deformations requires advanced topological arguments, as compared with the classical Dirichlet boundary value problems. Before further reflections let us take a brief look at one more example. \\
\textbf{An Example.} Given a pair of circular annuli $\,\mathbb X = \{z ; \,r < |z| < R\,\}\,$ and  $\, \mathbb Y = \{ w ; \,1 <|w| < \frac{1}{2} \,( R  + R^{-1} ) \,\}\,$,  $\, r < 1 < R\,$, the energy-minimal deformation in  $\,\mathscr M_2(\mathbb X , \mathbb Y) $, unique up to the rotations, is given by:
\begin{equation}\nonumber
 h(z)= \begin{cases}  \frac{z}{|z|}\, , &\; \; r<|z|\leqslant 1 \hskip0.5cm \\
\frac{1}{2} \left(z + \frac{1}{\overline{z}}\right) \, , &\; \; 1 < |z|< R \hskip1cm \textnormal{\footnotesize{( critical harmonic Nitsche map )}}
\end{cases}
\end{equation}
 see \cite{AIM, IOan}. Squeezing of  $\, \{ z\,; \;r<|z|\leqslant 1\,\} \subset \mathbb X\,$ into the inner circle of $\,\mathbb Y\,$ manifests itself.
Let us confront it with the solution of the classical Dirichlet problem under the same boundary values as $\,h\,$; that is,
\begin{equation}
 \hbar(z)= \begin{cases}  r^{-1} z \, , &\; \; \textnormal{for}\;\;|z| = r \\
\frac{1}{2} \left(1 + R^{-2}\right) z  \, , &\; \;\textnormal{for}\;\; |z| = R
\end{cases}
\end{equation}
The solution is a harmonic mapping

$$
\hbar(z) = A\, z + \frac{B}{\bar{z}} \;, \;\;\textnormal{where}\;\; A = \frac{R^2 - 2r + 1}{2(R^2 - r^2)}\;\;\textnormal{and} \;\;\; B= \frac{r(2 R^2 - rR^2 - r) }{2(R^2 - r^2)}\;
$$
 which, as expected, fails to be injective. The point is that, under this harmonic solution the annulus   $\,\mathbb X\,$ folds along the circle  $\,|z| = \sqrt{\frac{B}{A}}\,$.\\
  It is axiomatic in NE that folding should not occur under hyperelastic deformations. Uniform limits of hyperelastic deformations do not exhibit foldings as well. \\
It is a common struggle in mathematical models of Nonlinear Elasticity \cite{Anb, Bac, Ba3, Ba11, Ba1, Ba4, Cib} to establish existence of the energy-minimal deformations which  comply with the \textit{principle of no interpenetration of matter}. The phenomenon of squeezing a part of a domain, as described by monotone mappings, is inevitable. It is therefore reasonable to adopt  \textit{Monotone Sobolev Mappings} as legitimate deformations in the mathematical description of $\,2D\,$-elasticity (cellular mappings in higher dimensions).

 \bibliographystyle{amsplain}

\end{document}